\documentclass[11pt,reqno, oneside]{amsart}
\usepackage{pdfsync}  
\usepackage{geometry}
\usepackage{hyperref, fullpage}
\usepackage{diagbox}
\usepackage{subcaption, enumitem}
\usepackage{color}
\usepackage{amsmath}

\newtheorem{thm}{Theorem}[section]
\newtheorem{cor}[thm]{Corollary}
\newtheorem{lem}[thm]{Lemma}
\newtheorem{prop}[thm]{Proposition}

\theoremstyle{definition}
\newtheorem{defn}[thm]{Definition}

\theoremstyle{definition}

\theoremstyle{definition}
\newtheorem{question}[thm]{Question}

\theoremstyle{definition}
\newtheorem{obs}[thm]{Observation}

\theoremstyle{definition}
\newtheorem{ex}[thm]{Example}

\newcommand{\D}{\mathcal{D}}
\renewcommand{\S}{\mathcal{S}}
\newcommand{\A}{\mathcal{A}}
\newcommand{\B}{\mathcal{B}}

\DeclareMathOperator{\Av}{Av}
\DeclareMathOperator{\Type}{Type}

\title{Pattern-restricted permutations composed of 3-cycles}
\author{Kassie Archer}
\address[K. Archer]{University of Texas at Tyler, Tyler, TX 75799 USA}
\email{karcher@uttyler.edu}

\author{Christina Graves}
\address[C. Graves]{University of Texas at Tyler, Tyler, TX 75799 USA}
\email{cgraves@uttyler.edu}

\begin{document}

\maketitle

\begin{abstract}
In this paper, we characterize and enumerate pattern-avoiding permutations composed of only 3-cycles. In particular, we answer the question for the six patterns of length 3. We find that the number of permutations composed of $n$ 3-cycles that avoid the pattern 231 (equivalently 312) is given by $3^{n-1}$, while the generating function for the number of those that avoid the pattern 132 (equivalently 213) is given by a formula involving the generating functions for the well-known Motzkin numbers and Catalan numbers. The number of permutations composed of $n$ 3-cycles that avoid the pattern 321 is characterized by a weighted sum involving statistics on Dyck paths of semilength~$n$.
\end{abstract}

\section{Introduction}

%\todo{cite Smith and Bona, new paper}

%{\color{red} Possibly change the star notation to superscript $\{3\}$??}

Let us denote by $\S_n$ the set of permutations on $[n]=\{1,2,\ldots,n\}$. We write a permutation $\pi \in \S_n$ in its \emph{one-line notation} as $\pi=\pi_1\pi_2\ldots\pi_n$ where $\pi_i = \pi(i)$. We can also consider a permutation $\pi \in \S_n$ written in its \emph{cycle notation} as a product of disjoint cycles. For example, the permutation $\pi = 2415376$ written in its one-line notation can also be expressed as $\pi = (1, 2, 4, 5, 3)(6, 7)$. In a permutation's cycle notation, we say a cycle is a $k$-cycle if it is a cycle composed of $k$ elements. For example, $\pi = (1, 2, 4, 5, 3)(6, 7)$ is composed of one 5-cycle and one 2-cycle.

A permutation $\pi\in\S_n$ \emph{avoids} the pattern $\sigma\in\S_k$ if there is no subsequence of $\pi$ that appears in the same relative order as $\sigma$. For example, $\pi = 2415376$ avoids the pattern $321$ since there is no subsequence of length 3 in $\pi$ that appears in decreasing order. We denote by $\Av_n(\sigma)$ the set of permutations in $\S_n$ that avoid the pattern $\sigma$. Additionally, we say that a permutation $\pi \in \S_n$ avoids the set of patterns $\{\sigma_1, \sigma_2, \ldots, \sigma_k\}$ if for each $1\leq i \leq k$, $\pi$ avoids $\sigma_i$; we denote by $\Av_n(\sigma_1, \sigma_2, \ldots, \sigma_k)$ the set of such permutations.

Simion and Schmidt began the enumerative study of pattern-avoiding permutations in \cite{SS85}. In the same paper, they enumerated certain pattern-avoiding involutions, i.e.\ those permutations that are their own algebraic inverse. These permutations can be characterized as those that are composed of only 1- and 2-cycles. Since then, pattern-avoiding involutions have been further enumerated by number of fixed points \cite{DRS07, GM02}. %should I specifically mention fixed-point free involutions? Statements of previous theorems in terms of the superscript notation?
Recently, papers have posed the question of enumerating certain pattern-avoiding permutations whose algebraic cube is itself. For instance, in \cite{BS19}, B\'{o}na and Smith asked, ``How many 132-avoiding permutations of length $n$ are there in which each cycle length is 1 or 3?"  This question remains open, but in \cite{BD19}, the authors answer the analogous question for the pattern 231.

In this paper, we address related questions for the set of permutations on $[3n]$ whose cycle decomposition consists only of 3-cycles which we denote by $\S_{3n}^\star$. For example, $\S_3^\star = \{312, 231\}$ and for $n\geq 1$, $|\S_{3n}^\star| = \frac{(3n)!}{n!3^n}$.  We will denote those permutations in $\S_{3n}^\star$ that avoid the pattern $\sigma$ by $\Av_{3n}^\star(\sigma)$. 
Note that $\pi \in\Av_{3n}^\star(\sigma)$ if and only if $\pi^{-1} \in \Av_{3n}^\star(\sigma^{-1})$, where $\pi^{-1}$ is the algebraic inverse of the permutation $\pi$. Similarly,  $\pi \in\Av_{3n}^\star(\sigma)$ if and only if $\pi^{rc} \in \Av_{3n}^\star(\sigma^{rc})$ where $\pi^{rc}$ is the reverse-complement of $\pi$. For these reasons, we have the equalities
\[
|\Av_{3n}^\star(231)| = |\Av_{3n}^\star(312)| \quad \text{ and } \quad |\Av_{3n}^\star(132)| = |\Av_{3n}^\star(213)|. 
\]

\begin{figure}
\centering
\begin{tabular}{|c|c|c|}
\hline
Avoiding & Sequence & Theorem\\
\hline
231 & $1, 3, 9, 27, 81, \ldots$ & Theorem \ref{Thm231} \\ \hline
132 & 2, 8, 36, 170, 824 \ldots & Theorem \ref{Thm132}\\ \hline
321 & 2, 10, 60, 388, 2606 \ldots   &Theorem \ref{Thm321} \\ \hline
123 & $2, 6, 0, 0, 0, \ldots$ &Theorem \ref{Thm123} \\ \hline
\end{tabular}
\end{figure}

In Section~\ref{231}, we characterize and enumerate the 231-avoiding permutations that are composed of $n$ 3-cycles. As stated in Theorem \ref{Thm231}, we show that there are $3^{n-1}$ such permutations. Section~\ref{132} characterizes and enumerates the 132-avoiding permutations in $\S_{3n}^\star$. The enumeration is given in Theorem \ref{Thm132} and consists of a sum over compositions of $n$ of certain Catalan and Motzkin numbers. Theorem \ref{Thm132} also gives the generating function for these 132-avoiding permutations in terms of $c(x)$ and $m(x)$, the generating functions for the Catalan numbers and Motzkin numbers, respectively. %The generating function for $|\Av_{3n}^\star(132)|$ is
%\[ \frac{2(c(x)-1)m(c(x)-1)}{1 - (c(x)-1)m(c(x)-1)}. \]
In Section~\ref{321}, we give the enumeration of $\Av_{3n}^\star(321)$ in terms of a weighted sum involving statistics on Dyck paths of semilength $n$ in Theorem \ref{Thm321}.  Finally, we conclude by addressing the pattern 123 and pairs of patterns  in Section~\ref{other}.

\section{Avoiding $231$}\label{231}
We begin by enumerating the set $\Av_{3n}^\star(231)$.  This section is devoted to proving the following theorem.

\begin{thm}\label{Thm231}
For all $n\geq 1$, $$|\Av_{3n}^\star(231)| = 3^{n-1}.$$
\end{thm}

Since the permutations in $\Av_{3n}^\star(231)$ avoid $231$ or $(1,2,3)$ in cycle form, they must be composed only of cycles of the form $312$ or  $(1,3,2)$ in cycle form. To prove Theorem~\ref{Thm231}, we will find a bijection between these permutations and words of length $n-1$ on the alphabet $\{E,L,R\}$. 
 %Because these operations involve relabeling elements in a permutation, we give an informal definition and example first and follow with the formal definition of the operations.
To this end, we define three operations that take in a permutation in $\Av_{3n}^\star(231)$ and return a permutation in $\mathcal{S}^*_{3(n+1)}$ by specifying positions of inserted elements. %$\Av_{3(n+1)}^\star(231)$. For a given permutation $\pi$ in $\Av_{3n}^\star(231)$, we will give three new permutations in $\mathcal{S}^*_{3(n+1)}$ by specifying positions of inserted elements. ($\Av_{3(n+1)}^\star(231)$ )
%These inserted elements comprise a 3-cycle of the form $312=(1,3,2)$, so their positions also determine their values. The other elements will be in the same relative order as the original permutation.

\begin{defn}\label{def ELR} Let $\pi \in \Av_{3n}^\star(231)$, and let $a< b < 3n$ be the positions of the elements in the 3-cycle containing $3n$. Define operations $E, L,$ and $R$, by inserting three new elements into $\pi$ which will comprise of a 3-cycle of the form $312=(1,3,2)$ as follows:
\begin{itemize}
\item $E(\pi)$ has the three new elements all at the \emph{end} of $\pi$.
\item $L(\pi)$ has the new elements to the \emph{left} of $a$, to the left of $3n$, and to the right of $b$, and
\item $R(\pi)$ has the new elements to the \emph{right} of $a$, to the left of $3n$, and to the right of $b$.
\end{itemize}
The positions of the new elements also determine their values (as stated in the forthcoming observation), while the remaining elements are in the same relative order as $\pi$ and are determined by their positions.
\end{defn}
%
%We note that the new elements that are inserted are:
%\begin{itemize}
%\item $3n+3, 3n+1$, and $3n+2$ for $E(\pi)$
%\item $
\begin{obs}\label{obs1} Let $\pi \in \Av_{3n}^\star(231)$ with $a < b < 3n$ the positions of the elements in the 3-cycle containing $3n$. Then
\begin{itemize}
\item The elements inserted into $\pi$ to form $E(\pi)$ are $\{3n+1, 3n+2, 3n+3\}$.
\item The elements inserted into $\pi$ to form $L(\pi)$ are $\{a, b+1, 3n+3\}$.
\item The elements inserted into $\pi$ to form $R(\pi)$ are $\{a, b+2, 3n+3\}$.
\item In $R(\pi)$ and $L(\pi)$, the element to the immediate right of $3n+3$ is $3n+2$.
\end{itemize}
\end{obs}

We demonstrate these operations with an example. 

%\begin{defn}
%Given a permutation $\pi$ in $\Av_{3n}^\star(231)$, we define operations $E, L,$ and $R$ as follows.  Let $(a, 3n, b)$ be the cycle in $\pi$ containing $3n$. 
%\begin{itemize}
%\item Define $E(\pi)$ by appending the cycle $(3n+1, 3n+3, 3n+2)$. Clearly the permutation $E(\pi)$ is composed of only 3 cycles and avoids 231 since $\pi$ does. 
%\item Define $L(\pi)$ by appending the cycle $(a, 3n+3, b+1)$ and relabeling the remaining terms, preserving order. 
%\item Define $R(\pi)$ by appending the cycle $(a, 3n+3, b+2)$ and relabeling the remaining terms, preserving order. 
%\end{itemize}
%\end{defn}
\begin{ex}
Let  $$\pi = 3\ 1\ 2\ \mathbf{12}\ 11\ 10\ 5\ \mathbf{4}\ 6\ 9\ 7 \ \mathbf{8} $$ be a permutation in $\Av_{12}^\star(231)$ with the 3-cycle containing 12 bolded. Following the notation in Definition~\ref{def ELR}, we have $a=4$ and $b=8$.  Then the three new permutations, $E(\pi), L(\pi),$ and $R(\pi)$, will be of the forms
\begin{align*}
&3\ 1\ 2\ \mathbf{12}\ 11\ 10\ 5\ \mathbf{4}\ 6\ 9\ 7 \ \mathbf{8}\ *\  *\ \ * , \\
&3\ 1\ 2\ *\ \mathbf{12}\ 11\ 10\ 5\ *\ \mathbf{4}\ 6\ 9\ 7 \ \mathbf{8}\ *,  \text{ and}\\
&3\ 1\ 2\ *\ \mathbf{12}\ 11\ 10\ 5\ \ \mathbf{4}\ * 6\ 9\ 7 \ \mathbf{8}\ *,
\end{align*}
where $*$ indicates the position of the new 3-cycle. The positions will also give us the values in the new 3-cycle and those values will appear in the relative order $312$, as seen in Observation~\ref{obs1}. The original elements must be shifted appropriately. For the first new permutation, $E(\pi)$, the new elements are in positions 13, 14, and 15, and therefore no new relabeling of the other elements is necessary. For the other two new permutations relabeling is necessary (since their positions have been shifted in some cases), and we have the following new permutations with the inserted elements bolded. 
$$E(\pi) = 3\ 1\ 2\ 12\ 11\ 10\ 5\ 4\ 6\ 9\ 7 \ 8 \ \mathbf{15 \ 13 \ 14}$$
$$L(\pi) = 3\ 1\ 2\ \mathbf{15} \ 14\ 13\ 12\ 6\ \mathbf{4} \ 5 \ 7\ 11\ 8 \ 10 \ \mathbf{9} $$
$$R(\pi) = 3\ 1\ 2\ \mathbf{15} \ 14\ 13\ 12\ 6\ 5 \ \mathbf{4} \ 7\ 11\ 8 \ 9 \ \mathbf{10}$$
\end{ex}

%\begin{ex}
%Suppose that $$\pi = 3\ 1\ 2\ 12\ 11\ 10\ 5\ 4\ 6\ 9\ 7 \ 8 = (1,3,2)(4,12,8)(5,11,7)(6, 10,9).$$
%Then to get $E(\pi)$, we append $(13, 15, 14)$ which corresponds to placing $15 \ 13 \ 14$ at the end of $\pi$ and we obtain
%\begin{align*}
%E(\pi) &= 3\ 1\ 2\ 12\ 11\ 10\ 5\ 4\ 6\ 9\ 7 \ 8 \ 15 \ 13 \ 14 \\ &= (1,3,2)(4,12,8)(5,11,7)(6, 10,9)(13, 15, 14).
%\end{align*}
%To obtain $L(\pi)$, we append $(4, 15, 9)$ and adjust the other entries accordingly. This corresponds to inserting our new 3-cycle into $\pi$ before $3n$, before $a$, and after $b$. We obtain
%\begin{align*}
%L(\pi) &= 3\ 1\ 2\ 15 \ 14\ 13\ 12\ 6\ 4 \ 5 \ 7\ 11\ 8 \ 10 \ 9 \\ &= (1,3,2)(4,15,9)(5,14,10)(6,13,8)(7, 12,11).
%\end{align*}
%To obtain $R(\pi)$, we append $(4, 15, 10)$ and adjust the other entries accordingly. This corresponds to inserting our new 3-cycle into $\pi$ before $3n$, after $a$, and after $b$. We obtain
%\begin{align*}
%R(\pi) &= 3\ 1\ 2\ 15 \ 14\ 13\ 12\ 6\ 5 \ 4 \ 7\ 11\ 8 \ 9 \ 10 \\ &= (1,3,2)(4,15,10)(5,14,9)(6,13,8)(7, 12,11).
%\end{align*}
%\end{ex}

To prove Theorem \ref{Thm231}, we need to show that $E(\pi), L(\pi),$ and $R(\pi)$ all avoid $231$, and that every permutation in $\Av^\star_{3(n+1)}(231)$ can be obtained from an $E, L$, or $R$ operation performed on a permutation from $\Av^\star_{3n}(231)$.
\begin{lem}\label{lem:onto}
 Let $\pi \in \Av^\star_{3n}(231)$. Then $E(\pi), L(\pi), R(\pi) \in \Av^\star_{3(n+1)}(231)$.
\end{lem}

\begin{proof} Clearly, $E(\pi), L(\pi)$, and $R(\pi)$ are in $\S^\star_{3(n+1)}$, so we need only show that these permutations avoid $231$. We note that $E(\pi)$ avoids 231 since the new terms appear at the end of the permutation, are the largest three terms in the permutation, and appear in the order 312. 

Let us consider $L(\pi)$. By Observation~\ref{obs1}, the ``new'' elements added to $\pi$ were the values $\{a, b+1, 3n+3\}$ added in the order 312 in those given positions, where $a<b<3n$ were the positions of the 3-cycle containing $3n$ in $\pi$. The original 3-cycle containing $3n$, upon relabeling, becomes $(a+1, 3n+2, b+2)$. Notice that all new elements are adjacent to the terms in the 3-cycle containing the largest element of $\pi$. Therefore, any new 231 pattern must use adjacent  terms from the new cycle $(a, 3n+3, b+1)$ and the cycle $(a+1, 3n+2, b+2)$ since otherwise an occurrence of 231 using new elements would correspond to an occurrence of 231 in $\pi$. In $L(\pi)$, these terms occur in the order: 
$$
\ldots (3n+3)(3n+2)\ldots a(a+1) \ldots (b+2)(b+1).
$$
The terms $(3n+3)(3n+2)$  cannot be part of a 231 pattern, nor can the terms $(b+2)(b+1)$. It remains to check that $a(a+1)$ is not part of a 231 pattern. To check this, we must show that all elements after $a+1$ are greater than $a$. This means that in $\pi$, we must have that all elements after $a$ are greater than $a$. However, $a$ is the position of $3n$ in $\pi$ and all elements of $\pi$ to the left of $3n$ must be less than all elements to the right of $3n$ in $\pi$. Therefore $a$ is the smallest element to the right of $3n$. 

For $R(\pi)$, we similarly have that the terms comprising the new 3-cycle and the original cycle containing $3n$ in $\pi$ (upon relabeling) appear in the order:
$$
\ldots (3n+3)(3n+2)\ldots (a+1)a \ldots (b+1)(b+2).
$$
Each consecutive pair clearly cannot be part of a 231 pattern, and so no new 231 pattern was added. 
\end{proof}

We now show that every permutation  in $\Av^\star_{3n}(231)$ can be obtained from an $E, L$, or $R$ operation.
\begin{lem} \label{lem:invertible}
Let  $\pi \in \Av^\star_{3(n+1)}(231)$. Then there exists a unique $\tau \in  \Av^\star_{3n}(231)$ so that either $E(\tau) = \pi$, $L(\tau) = \pi$, or $R(\tau)=\pi$.
\end{lem}

\begin{proof}
Let $A< B< 3n+3$ be the positions of the 3-cycle in $\pi$ containing the element $3n+3$. 
If $B = 3n+2$, then we must also have that $A = 3n+1$. In this case, let $\tau$ be the first $3n$ elements of $\pi$. Clearly, $\tau \in \Av^\star_{3n}(231)$ and $E(\tau) = \pi$. Otherwise, it is enough to show that the cycle containing $3n+2$ consists of the elements $\{A+1, B+1, 3n+2\}$ or $\{A+1, B-1, 3n+2\}$. In the first case we will see that $\pi =L(\tau)$, and in the second case $\pi = R(\tau)$ for some unique $\tau \in  \Av^\star_{3n}(231)$.

Let $x<y<3n+2$ be the positions of the 3-cycle containing $3n+2$. First note that since $\pi$ avoids $231$, everything to the left of $3n+3$ (in position $A$) must be less than everything to the right of $3n+3$. Therefore, $A, B, x, y,$ and $3n+2$ must all appear to the right of $3n+3$.  Notice that this also implies that $A$ is the smallest element that occurs after $3n+3$. 

In position $x$, we have $3n+2$ since this 3-cycle must form a 312 pattern. If $A$ appears before $3n+2$ (that is, if $B<x$), then the subsequence $xyB$, which occurs in positions $y<3n+2<3n+3$, is an occurrence of 231. On the other hand, if $A$ appears after $3n+2$ (that is, $B>x$) and $x\neq A+1$, then $\pi_{A+1}(3n+2)A$ is an occurrence of $231$. Therefore, we must have that $x = A+1$. 

It remains to show that $y = B-1$ or $B+1$. Suppose that $y<B$ and $y\neq B-1$. Then $\pi_y\pi_{B-1}\pi_B = (A+1)\pi_{B-1}A$ is clearly an occurrence of 231. Now suppose instead that $y>B$ and what $y \neq B+1$. Then $(B+1)yB$ is an of 231 in $\pi$ since $yB$ occurs in the last two positions of $\pi$. Therefore it must be the case that $y=B-1$ or $B+1$. 
%Therefore, we have that $\pi$ is either of the form
%$$
%\red{\alpha}(3n+3)\red{\beta} (3n+2)\red{\gamma} x\red{\delta} A \red{\epsilon} yB
%$$
%or 
%$$
%\red{\alpha}(3n+3)\red{\beta} (3n+2)\red{\gamma} A\red{\delta} x \red{\epsilon} yB.
%$$
%Since $A$ is the smallest element to the right of $3n+3$ and $3n+2$ is the largest element to the right of $3n+3$, 
\end{proof}

We can now prove that the number of 231-avoiding permutations composed of only 3-cycles is equal to $3^{n-1}$. 
\begin{proof}[Proof of Theorem~\ref{Thm231}]
Consider the following map  $\varphi$ from the set $\mathcal{W}_{n-1}^3$ of words of length $n-1$ on letters $\{E, L, R\}$ to $\Av^\star(231)$. Let $w=w_1\ldots w_{n-1} \in\mathcal{W}_{n-1}^3$. Start with the permutation $\pi^1 = (1,3,2) = 312$. For $2\leq i \leq n$, obtain $\pi^i$ from $\pi^{i-1}$ by taking $w_{i-1}(\pi_{i-1})$. Then take $\varphi(w) = \pi^n$. 

%Consider also the map $\psi$ from the set $\Av_{3n}^\star(231)$ to  the set $\mathcal{W}_{n-1}^3$ given by letting $\pi^n = \pi \in \Av_{3n}^\star(231)$ and obtaining $\pi^{i-1}$ from $\pi^i$ by removing the cycle containing $3i$ and relabelling as necessary, preserving order. If the cycle removed was of the form $(3i-2, 3i, 3i-1)$, then record $w_{i-1}=E$. If not, then for the cycle of the form $(a, 3i, b)$, record $w_{i-1}=L$ if $\pi^i_{b+1}<\pi^i_{b-1}$ and $w_{i-1}=R$ if $\pi^i_{b+1}>\pi^i_{b-1}$.

%It remains to show that in fact $\varphi$ and $\psi$ are inverses of each other to establish that they are indeed bijections. 

By repeatedly applying Lemmas \ref{lem:onto} and \ref{lem:invertible}, we can see that this map is onto and invertible and is thus a bijection. The result follows. 
\end{proof}

\section{Avoiding $132$}\label{132}

%Recall that $\Av_{3n}^\star(132)$ denotes the set of permutations on $3n$ elements that avoid the pattern 132. If $\sigma \in \Av_{3n}^\star(132)$, then all cycles in $\sigma$ are either the pattern 312 or 231. We begin by examining the permutations in $\Av_{3n}^\star(132)$ that contain only 3-cycles with the pattern 312.  To this end,  let us denote by $\A_{3n} \subseteq \Av_{3n}^\star(132)$  the set of permutations on $3n$ elements that avoid the pattern 132 and contain only 3-cycles of the form 312.

For any permutation $\pi \in \Av_{3n}^\star(132)$, each of the $n$ 3-cycles correspond to an occurrence of either the pattern 312 or the pattern 231. If a cycle corresponds to the 312 pattern, we will say the cycle is of the form 312 and similarly, if it corresponds to a 231 pattern, we will say it is of the form 231. In this section, we begin by studying and enumerating those permutations $\pi \in \Av_{3n}^\star(132)$ where all $n$ 3-cycles in $\pi$ are of the form 312. We will ultimately use this result to enumerate all of $\Av_{3n}^\star(132)$ at the end of the section.

To this end,  let us denote by $\A_{3n} \subseteq \Av_{3n}^\star(132)$  the set of permutations on $3n$ elements that avoid the pattern 132 and contain only 3-cycles of the form 312. For each $\pi\in\A_{3n}$, let us define the sets $T_1, T_2, T_3$ and a word $W$ as below. 

\begin{defn} \label{def:T} Let $\pi = \pi_1\pi_2\cdots \pi_{3n} \in \A_{3n}$.
\begin{itemize}
\item Define the sets $T_1(\pi), T_2(\pi),$ and $T_3(\pi)$, or simply $T_1, T_2, T_3$, to be a partition of $[3n]$ where $\pi_j \in T_i$ if $\pi_j$ is the ``$i$" in the 312 pattern realized by its cycle, that is, if $\pi_j$ is the $i^{\text{th}}$ smallest number in its cycle. %or equivalently, if $i= |\{t\in \{j, \pi_j, \pi_{\pi_j}\}: t\leq \pi_j\}|$.
\item Define $W(\pi)=w_1w_2\cdots w_{2n}$, or simply $W$, to be a word where $w_i = j$ if $\pi_{n+i} \in T_j$ for $1 \leq i \leq 2n$ and $1\leq j\leq 3$.
\end{itemize}
\end{defn}

%he sets $T_i$ and the word $w$ have specific properties which we state now.

%{\color{red} I changed (a) to say $T_1$ (rather than $T_3$) and (c) to say for $i,j$ in $T_3$ rather than in $T_2$. $T_3$ should be the positions of the elements in $T_2$}

 Example \ref{ex:Ts} shows $T_1, T_2,$ and $T_3$ and the associated word $W$ for a given permutation of length 18. Certain properties of these sets and the word $W$  are outlined in the following lemma.

\begin{lem} \label{PiConditions}
Let $\pi = \pi_1\pi_2\cdots \pi_{3n} \in \mathcal{A}_{3n}$. Then, 
\begin{enumerate}[label=(\alph*)]
\item\label{item:T - 1} $T_1 = \{1,2,\ldots, n\}$, or equivalently $T_3 = \{\pi_1, \pi_2, \ldots, \pi_n\}$;
\item\label{item:T - 2} $W$ is a Dyck word on 1s and 2s; and 
\item\label{item:T - 3} For $i,j \in T_3$, if $i< j$ then $\pi_i < \pi_j$, or equivalently, the elements of $T_2$ appear in increasing order in~$\pi$.
\end{enumerate}
\end{lem}

%\todo{Check that the wording of condition c is okay.  I added the second part.  Also, in the proof below, would it make more sense for $\beta=\beta_3\beta_1\beta_2$? I don't care but I was confused the first time I read it. - C}

\begin{proof} First note that the statements in \ref{item:T - 1} are equivalent since every cycle in $\pi$ has form 312, and thus the positions of elements in $T_3$ are given by the elements in $T_1$. Let us prove \ref{item:T - 1} by contradiction. To this end, suppose that $\pi_k \in T_1$ for some $2 \leq k \leq n$.  The 3-cycle associated with this element must contribute the 312 pattern $k\pi_k\pi_{\pi_k}$. In particular, since $\pi_k$ is the smallest element in its cycle, we must have $\pi_k<k$. Because there are $n$ cycles of the form 312 in $\pi$, there must be some cycle of the form 312  that we can call $\beta = \beta_3\beta_1\beta_2$, whose first element is at a position greater than $k$, that is, $\beta_3=\pi_j$ for some $j>k$. But if its first position is greater than $k$, then $\beta_i > k$ for all $i$.  However, we then have that $\pi_k \beta_3 \beta_1$ is a 132 pattern which is a contradiction. Since $\pi_k \notin T_1$ for all $1 \leq k \leq n$, there also cannot be a $\pi_k \in T_2$ for any $1 \leq k \leq n$ because every ``2" in a cycle of the form 312 must be preceded by a ``1." Thus, $\pi_k \in T_3$ for $1 \leq k \leq n$.

Condition \ref{item:T - 1} implies that $\{\pi_{n+1}, \ldots, \pi_{3n}\} = T_1\cup T_2$, and thus $W$ is a word on 1s and 2s. Furthermore, for each cycle of the form 312   in $\pi$, the smallest element comes before the middle element, so for each $2$, there is a corresponding $1$ that appears before it. Thus $W$ is a Dyck word and \ref{item:T - 2} holds.

Finally, suppose toward a contradiction that \ref{item:T - 3} does not hold and there are two elements $i, j \in T_3$ with $i<j$ and $\pi_i>\pi_j$. Because $\pi_i$ is the ``2" in its cycle of the form 312, the element $\pi_{\pi_i}$ appears before $\pi_i$ and is the ``1'' in the cycle. By \ref{item:T - 1}, $\pi_{\pi_i}$ is also less than both $\pi_i$ and $\pi_j$. Therefore, $\pi_{\pi_i}\pi_i\pi_j$ is a 132 pattern which is a contradiction. 
\end{proof}

We illustrate Definition \ref{def:T} and  Lemma~\ref{PiConditions} with the following example.
\begin{ex}\label{ex:Ts}
Let 
\begin{align*}
 \pi& = 18\ 16\ 14\ 15\ 12\ 11\ 6\ 5\ 3\ 4\ 7\ 8\ 2\ 9\ 10\ 13\ 1\ 17\\&= (1,18,17)(2,16,13)(3,14,9)(4,15,10)(5,12,8)(6,11,7).
 \end{align*}
 Then
 \begin{align*}
  T_1 &= \{1,2,3,4,5,6\},\\
 T_2 &= \{7,8,9,10,13,17\},\\
   T_3 &=\{11,12,14,15,16,18\},
 \end{align*}
 and \[ W = 111122122212\] is the corresponding Dyck word. 
 Furthermore, the elements of $T_2$ are in increasing order in $\pi$.
 \end{ex}
 
We now have that each permutation in $\A_{3n}$ is associated with a Dyck word on 1s and 2s of length $2n$; however, this is not a bijective map.  For instance, both the permutation $561234$ and $652134$ are associated with the Dyck word $1122$. To count the number of permutations in $\A_{3n}$ associated with each Dyck word, we need more information about the Dyck word. The following definition assigns a ``type" to a given Dyck word.

% 
%%\[ \pi' = {\color{red} 5\ 6\ 4\ 3}\ 8\ 7\ {\color{red}2}\ 9\ 10\ 13\ {\color{red}1}\ 17,\] and
%the corresponding Dyck word is \[ W = {\color{red} rrrr}bb{\color{red} r}bbb{\color{red} r}b.\] The 1st red is followed by red, but the 1st black is followed by black. So 1 is not in $S$. The 2nd red is followed by red and the 2nd black is followed by red, so 2 is in $S$. The 4th red is followed by black, so $4 \in S$. Likewise, $5, 6 \in S$. So $S = \{2,4,5,6\}$ where we denote $s_1 = 2, s_2=4, s_3=5,$ and $s_4=6$. Then
%\begin{align*}
%x_1 &= s_1 = 2\\
%x_2 &= s_2-s_1 = 2\\
%x_3 &= s_3-s_2 = 1\\
%x_4 &= s_4-s_3= 1.
%\end{align*}
% Thus $\Type(W) = (2,2,1,1)$.
%\end{ex}
%

\begin{defn} \label{TypeW}
Let $W = w_1\cdots w_{2n}$ be a Dyck word on 1s and 2s of length $2n$. Let $S_1$ and $S_2$ be sets containing the positions of the 1s and 2s, respectively, and write $S_1= \{u_1 < u_2 < \cdots < u_n\}$ and $S_2= \{v_1 < v_2 < \cdots < v_n\}$. % be a partition of $[2n]$ where $i \in S_j$ if $d_i=j$.
Define $S \subset [n]$ so that $i \in S$ if and only if either
\begin{itemize}
\item $v_i+1 \in S_1$, or
\item $u_i + 1 \in S_2$.
\end{itemize}
In other words, $i \in S$ if and only if either
\begin{itemize}
\item the $i^{\text{th}}$ 2 in $W$ is followed by 1, or
\item the $i^{\text{th}}$ 1 in $W$ is followed by 2.
\end{itemize}
Label the elements of $S$ as $s_1, s_2,\ldots, s_k$ in increasing order. Let $x_1 = s_1$ and $x_i = s_i - s_{i-1}$ for $2 \leq i \leq k$. Then we say $W$ is of \emph{type} $(x_1, x_2, \ldots, x_k)$, and write $\Type(W) = (x_1, x_2, \ldots, x_k)$.
\end{defn}

It is not hard to see that $\Type(W)= (x_1, x_2, \ldots, x_k)$ is a composition of $n$. The reasoning is that since the $n^{\text{th}}$ 1 in a Dyck word on 1s and 2s of length $2n$ is followed by a 2, it must be the case that the largest element in $S$ is $n$. Because $\Type(W)$ is formed from the differences of elements in $S$, then $\Type(W)$ is a composition of its largest element. %Indeed, since $\Type(W)$ is formed by listing the differences in $S$, we see that $\Type(W)= (x_1, x_2, \ldots, x_k)$ is a composition of $n$.
%
%{\color{red} I think I want to change that last remark to a lemma and also include the fact that the last ``block" of 1s has to be length $x_k$ so I can use it later.}
%
%\begin{lem} \label{LastBlock}
%something here
%\end{lem}
%

%\begin{defn}
%Let $\pi = \pi_1\cdots \pi_{3n} \in \mathcal{C}_{3n}(312)$ and let $\pi' = \pi_{n+1} \pi_{n+2}, \ldots \pi_{3n}$ $D$ be the Dyck word colored with red and black. Let $S \subset [n]$ so that $i \in S$ if and only if either
%\begin{itemize}
%\item the $i$th red is followed by red and the $i$th black is followed by red, or
%\item the $i$th red is followed by black.
%\end{itemize}
%Label the elements of $S$ as $s_1, s_2\ldots s_k$ in increasing order.  Then $\pi$ is of Type
%\[ \Type(\pi) = (x_1, x_2, \ldots x_n) \]
%where $x_1 = s_1$ and $x_i = s_i - s_{i-1}$ for $2 \leq i \leq k$.
%\end{defn}

%\todo{We might also say $\Type(\pi) = (2,2,1,1)$???????}

\begin{ex}\label{ex:type}
Consider the Dyck word $W = 111122122212$  from Example~\ref{ex:Ts}. Here, the positions of the 1s and 2s are given by $S_1 = \{1,2,3,4,7,11\}$ and $S_2 = \{5,6,8,9,10,12\}$, respectively. Therefore, $2\in S$ since $v_2+1 = 6+1 = 7\in S_1$ and $4,5,6\in S$ since $u_4+1 = 5, u_5+1= 8, u_6+1=12 \in S_2$. The set $S$ in Definition \ref{TypeW} is $S = \{2,4,5,6\}$. The differences of successive elements in $S$ form the type of $W$, and thus $\Type(W) = (2,2,1,1)$. Note that $(2,2,1,1)$ is a composition of 6.
\end{ex}

It turns out that every composition of $n$ can be realized as a type for some Dyck path. Furthermore, we can count Dyck words of a given type by using the well-known Motzkin numbers which have numerous combinatorial interpretations.  A \emph{Motzkin path of length $n$} is a path from $(0,0)$ to $(n,0)$ consisting of up step $(1,1)$, flat steps $(1,0)$, or down steps $(1,-1)$, that never dips below the $x$-axis.  The corresponding \emph{Motzkin word} to a Motzkin path is the word consisting of $U$s, $F$s, and $D$s for each up step, flat step, and down step, respectively. The number of Motzkin paths of length $n$ is the $n^{\text{th}}$ \emph{Motzkin number} $M_n$.

\begin{lem} \label{MotType} Let $X=(x_1, x_2, \ldots, x_k)$ be a composition of $n$. Then there are $M_{k-1}$ Dyck words of type $X$.% where $M_{k-1}$ is the $(k-1)^{\text{st}}$ Motzkin number.
\end{lem}

%{\color{red} I stopped here! -CG} 

\begin{proof}
We begin by creating a bijection between the set of Dyck words of length $2n$ of type $X$ and the set of Dyck words of length $2k$ of type $(1,1,\ldots,1)$ as follows.
Let $W$ be a Dyck word of length $2k$ of type $(1,1,\ldots,1)$. Replace the $i^{\text{th}}$  1 with $x_i$ 1s and the $i^{\text{th}}$ 2 with $x_i$ 2s. This clearly creates a Dyck word of length $2n$ of type $X$. This process is also invertible (by taking the $i^\text{th}$ block of 1s of length $x_i$ and replacing it with one 1 and doing the same thing for the 2s), and is thus a bijection.

Next, we create a bijection between the set of Dyck words of length $2k$ of type $(1, 1,\ldots, 1)$ and the set of Motzkin paths of length $k-1$ as follows. Let $W$ be a Dyck word of length $2k$ of type $(1,1,\ldots,1)$. Create a Motzkin word $M= m_1m_2\ldots m_{k-1}$ by the following procedure:
\begin{itemize}
\item If the $i^{\text{th}}$ 1 and $i^{\text{th}}$ 2 are both followed by 1, then set $m_i=U$.
\item If the $i^{\text{th}}$ 1 and $i^{\text{th}}$ 2 are both followed by 2, then set $m_i = D$.
\item %If the $i^{\text{th}}$  1  is followed by 2, and the $i^{\text{th}}$ 2 is followed by 1,  then 
Otherwise, set $m_i = F$.
\end{itemize}
Before showing that this procedure results in a Motzkin path, we give an equivalent condition for setting $m_i=F$. Note that a Dyck word is of type $(1,1,\ldots,1)$ if and only if there is no $i$ so that the $i^{\text{th}}$ 1 is followed by 1 and the $i^{\text{th}}$ 2 is followed by 2. Therefore, the third bullet point above could be written as: 
\begin{itemize}
\item If the $i^{\text{th}}$ 1 is followed by 2 and the $i^{\text{th}}$ 2 is followed by 1,  then set $m_i=F$. 
\end{itemize}

%{\color{red} I didn't follow this paragraph so it probably needs to be redone.}

%Let us now show that we indeed get a Motzkin word from this procedure. Since the Dyck word $W$ has an equal number of 1s and 2s, the resulting path has an equal number of $U$s and $D$s. Additionally, since the number of 1s always weakly exceeds the number of 2s at any given point in $W$, the number of $U$s must weakly exceed the number of $D$s at any point in the construction of $M$. 
Let us now show that we indeed get a Motzkin word from this procedure. Because of the characterization above, we have that $m_i=U$ exactly when the $i^{\text{th}}$ 1 is followed by 1 and $m_i=D$ exactly when the $i^{\text{th}}$ 2 is followed by 2. Clearly, there are an equal number of occurrences of 11 as there are occurrences of 22 in any Dyck path. (Indeed, both are equal to $n-b$ where $b$ is the number of consecutive runs or blocks of 1s.) Also, since in a Dyck path the number of 1s weakly exceeds the number of 2s, similarly must the number of occurrences of 11 exceed the number of occurrences of 22. Thus in the word $M$, the number of $U$s must exceed the number of $D$s at any given point and is therefore a Motzkin word. 

Since in any Dyck word, you must start with 1, end with 2, and the last 1 must be followed by 2, it is straightforward to invert the process to obtain a Dyck path of type $(1,1,\ldots, 1)$ from a given Motzkin path. 
\end{proof}

In the following example, we find the four Dyck words with type $(2,2,1,1)$.

\begin{ex}
By Lemma \ref{MotType}, there are $M_3=4$ Dyck words of type $(2,2,1,1)$. The four Motzkin words of length three are
\[ UDF, FUD, UFD, \text{ and } FFF.\]
To find the Dyck word $W = w_1w_2w_3w_4w_5w_6w_7w_8$ of type $(1,1,1,1)$ from the Motzkin word $UDF,$ we start by setting $w_1=1$. Since $m_1=U$, the first 1 and the first 2 are followed by 1, so $w_2=1$. Since $m_2=D$, both the second 1 and the second 2 are followed by 2.  Thus, $w_3=2$ and $w_4=1$. Finally, because $m_3=F$, the third 1 must be followed by 2 and the third 2 is followed by 1.  Thus, $w_5=2$, $w_6=2$, and $w_7=1$. Finally, we set $w_8=2$ and obtain the Dyck word $W=11212212$.
A similar process works for the other Motzkin words, and we have following 4 Dyck words of type $(1,1,1,1)$:
\[ 11212212, 12112122, 11212122, 12121212.\]
Now to create the four corresponding Dyck words of type $(2,2,1,1)$, we replace the first 1 with two 1s, the first 2 with two 2s, the second 1 with two 1s, and the second 2 with two 2s:
\[ 111122122212, 112211122122, 111122122122, 112211221212.\]
These are the four Dyck words of type $(2,2,1,1)$.
\end{ex}

%{\color{red} I tried to put in an example that does everything above this so I think the next 2 examples can be deleted.}
%
%
%\begin{ex}
%Let us first see how to obtain a Dyck word of type $(1,1,\ldots,1)$ from a Dyck path of another type. Suppose $W = 1111122212211222$. Then $D$ is of type $(3,2,1,2)$. We can obtain a Dyck path of type $(1,1,1,1)$ by breaking $D$ into blocks of 1s and 2s of sizes 3, 2, 1, and 2, and reducing each block to size 1. Since: $$D = 111\; 11\; 222\; 1\; 22\; 11\; 2\; 22$$ we can obtain $D'$ of type $(1,1,1,1)$: $$W' = 11212122.$$
%\end{ex}
%
%\begin{ex}
%Here, let us see how to get a Motzkin path from a Dyck word of type $(1,1,\ldots, 1)$ and how to invert that process. Let $$W = 112112121122212122,$$ which is of type $(1,1,1,1,1,1,1,1,1)$. Then we can obtain $M = UFUDDUFD$ as the corresponding Motzkin path.
%
%To invert this, start with the Motzkin path $M = UFUDDUFD$. Every Dyck path starts with 1, so $d_1 = 1$. The fact that $m_1=U$ tells you that the first 1 is followed by a 1, so $d_1d_2=11$. Since $m_2=F$, the second 1 is followed by 2, so $d_1d_2d_3 = 112$. Since $m_1=U$, the first 2 is followed by 1, so $d_1d_2d_3d_4 = 1121$. Since $m_3=U$, the third 1 is followed by 1 so $d_1d_2d_3d_4d_5=11211$. This process can be continued to obtain the Dyck word $112112121122212122$. 
%\end{ex}

Recall that we are trying to enumerate the permutations in $\A_{3n}$. The following lemma counts the number of elements in $\A_{3n}$ associated to a given Dyck word $W$.

\begin{lem}\label{CatType}
Let $W$ be a Dyck word of length $2n$ with $\Type(W) = (x_1, x_2, \ldots x_k)$. Then the number of permutations in $\A_{3n}$ with corresponding Dyck word $W$ is
\[ C_{x_1}C_{x_2}\cdots C_{x_k} \]
where $C_k$ is the $k^{\text{th}}$ Catalan number.
\end{lem}

\begin{proof}
Let us start with a Dyck word $W=w_1w_2\cdots w_{2n}$ with $\Type(W) = (x_1, x_2, \ldots, x_k)$. Concurrent with Definition \ref{TypeW}, let $S_1= \{u_1 < u_2 < \cdots < u_n\}$ and $S_2 = \{v_1 < v_2 < \cdots < v_n\}$ be the set of positions of $W$ containing 1s and 2s, respectively. Now suppose that $\pi = \pi_1\pi_2\cdots\pi_{3n} \in \A_{3n}$ is a permutation with corresponding Dyck word $W$. We will find the conditions that $\pi$ must satisfy.

First, by Lemma \ref{PiConditions}, we know that $T_1(\pi) = \{1,2,\ldots, n\}$. The same lemma says that \[T_2(\pi) = \{u + n \mid u \in S_1\};\] these elements in $T_2(\pi)$ occur in increasing order in $\pi$ and are in positions \[T_3(\pi) = \{v + n \mid v \in S_2\}.\] We will determine where the elements in $T_1(\pi)$ can occur in $\pi$, which in turn determines $\pi$.

Consider again the Dyck word $W$. From the definition of type, we know that there are $k$ blocks of consecutive 1s in $W$ where the $i^{\text{th}}$ such block is of size $x_i$. Let us denote by $\mathcal{U}_i$ the set of positions in $\pi$ associated with the $i^{\text{th}}$ block of 1s, and by $\mathcal{V}_i$ the set of positions in $\pi$ associated with the $i^{\text{th}}$ block of 2s in $W$, so that $|\mathcal{U}_i|=|\mathcal{V}_i| = x_i$. Formally, we have that \[\mathcal{U}_i = \{u_j + n \mid 1 \leq j - (x_1 + \cdots + x_{i-1}) \leq x_i\},\] 
and
 \[\mathcal{V}_i = \{v_j + n \mid 1 \leq j - (x_1 + \cdots + x_{i-1}) \leq x_i\}.\] 
%In order to prove the lemma, we will first show that for every $i < j$, each element of $\pi$ with position 
%
%
%\vskip1cm
%
%
%
%
%%Let $\pi = \pi_1\pi_2\cdots\pi_{3n} \in \A_{3n}$ be a permutation associated to the 
%Let us start with a Dyck word $W$ with $\Type(W) = (x_1, x_2, \ldots x_k)$. Furthermore, suppose that the permutation $\pi = \pi_1\pi_2\cdots\pi_{3n} \in \A_{3n}$ has corresponding Dyck word $W$. We will find the conditions that $\pi$ must satisfy.
%
%By Lemma~\ref{PiConditions}, we know that $T_1(\pi) =\{1,2,\ldots, n\}$ and from the definition of type, we know that there are $k$ consecutive blocks of 1s in $W$ where the $i^{\text{th}}$ such block is of size $x_i$. Let us denote by $\mathcal{O}_i$ the set of positions in $\pi$ associated to the $i^{\text{th}}$ block of 1s in $W$ and by $\mathcal{T}_i$ the set of positions in $\pi$ associated to the $i^{\text{th}}$ block of 2s in $W$, so that $|\mathcal{O}_i|=|\mathcal{T}_i| = x_i$.
We will first show that for every $i<j$, each element of $\pi$ with position in $\mathcal{U}_i$ is greater than each element of $\pi$ with position in $\mathcal{U}_j$. Then we will show that the subsequence of $\pi$ given by positions in $\mathcal{U}_i$ can be taken to be order-isomorphic to any 132-avoiding permutation. %Since the elements of $T_2$ appear in increasing order in $\pi$ by Lemma \ref{PiConditions}, the order in which the numbers $1,2, \ldots, n$ appear in $\pi$ completely determine $\pi$.
Since the number of 132-avoiding permutations is given by the Catalan numbers, this would prove the lemma. 

Let us suppose for contradiction that there is some $i<j$ so that there exists $a\in\mathcal{U}_i$ and $b\in \mathcal{U}_j$ with $\pi_a<\pi_b$ and let us take it to be the largest such occurrence (with $j$ taken to be as large as possible, followed by the largest possible $i$). If there is some block of 2s between the $i^{\text{th}}$ block of 1s and the $j^{\text{th}}$ block of 1s in $W$ that we associate with $\mathcal{V}_m$ for some $m$, then there is an occurrence of 132 given by $\pi_a \pi_z\pi_b$ where $z$ is an element of $\mathcal{V}_m$.

If there is no block of 2s between the $i^{\text{th}}$ block of 1s and the $j^{\text{th}}$ block of 1s in $W$, then by the definition of type, there must be a block of 1s that occurs between the  $i^{\text{th}}$ block of 2s and the $j^{\text{th}}$ block of 2s in $W$, associated to $\mathcal{U}_r$ for some $r$. Since $i<j$, we must have $a<b$. Since $a<b$, $\pi_a<\pi_b$, and $a,b\in T_2$, we must also have that $\pi_{\pi_a}<\pi_{\pi_b}$. Furthermore, since $\pi_a$ and $\pi_b$ are taken have as large an index as possible, there must be some $t \in \mathcal{U}_r$ with $\pi_{\pi_a}<t<\pi_{\pi_b}$ (since $\pi_{\pi_a}$ is the position of $a$ and $\pi_{\pi_b}$ is the position of $b$). But $t\in T_2$ as well, so $t$ appears after $\pi_t$, and in particular after $\pi_{\pi_b}$. Therefore $\pi_{\pi_a}\pi_{\pi_b}t$ is an occurrence of 132. 

Finally, since the elements of $T_2$ occur in increasing order in the permutation $\pi$, we can make the following observations:
\begin{itemize}
\item  If we take any 132-avoiding permutation to be the pattern realized in positions in $\mathcal{U}_i$ for any $i$, then certainly there are no occurrences of 132 in $\pi_{n+1}\ldots \pi_{3n}$; and    
\item $\pi_1\ldots \pi_n$ will be composed of $k$  blocks, each of which has elements smaller than the block preceding it. Furthermore, the $i^{\text{th}}$ block is size $x_i$ and it forms the pattern whose inverse is the pattern realized by those positions in $\mathcal{U}_i$. 
\end{itemize}
With these two observations, we see that there will be no occurrences of 132 by taking the pattern in positions in $\mathcal{U}_i$ to be any 132-avoiding permutation. Since 132-avoiding permutations are enumerated by the Catalan numbers, the proof is complete. 
\end{proof}

The following example shows the permutations with  Dyck word $W = 111122122212$ from Example \ref{ex:Ts}.

\begin{ex}
Let $W = 111122122212$ which has type $(2,2,1,1)$ by Example \ref{ex:type}. $W$ has 4 consecutive blocks of 1s (and similarly 2s) where the 1st and 2nd blocks are size 2 and the 3rd and 4th blocks are size 1, as shown below:
\[
W = 11 \, \, 11 \, \, 22 \, \, 1 \, \, 22 \, \, 2 \, \, 1 \, \, 2.
\]
Notice that the 1s in $W$ are in positions $S_1 = \{1,2,3,4, 7, 11\}$. Thus, any permutation $\pi \in \A_{18}$ with Dyck word $W$ will have $T_2 = \{7,8,9,10, 13, 17\}$, and these elements occur in $\pi$ increasing order.  Using the notation in the proof of Lemma \ref{CatType}, we have
\[ \mathcal{U}_1 = \{7,8\},\ \mathcal{U}_2 = \{9,10\},\ \mathcal{U}_3=\{13\}, \text{ and }\  \mathcal{U}_4 = \{17\}.\]
Furthermore, if $i < j$ and $a \in \mathcal{U}_i$ and $b \in \mathcal{U}_j$, then $\pi_a > \pi_b$.  Combining these facts shows that $\pi$ must look like:
\[ 18\ 16\ \rule{.25cm}{.25mm}\ \rule{.25cm}{.25mm}\ \rule{.25cm}{.25mm} \ \rule{.25cm}{.25mm} \ \framebox{\strut \rule{.25cm}{.25mm} \ \rule{.25cm}{.25mm}} \ \framebox{\strut \rule{.25cm}{.25mm} \ \rule{.25cm}{.25mm}} \ 7\ 8\ \framebox{\strut 2}\ \ 9\ 10\ 13\ \boxed{\strut 1}\  17, \]
where the first block must be filled in with the numbers $\{5,6\}$ and the second block filled in with the numbers $\{3,4\}$ where each block avoids the pattern 132.  (The third and fourth blocks have been filled in with the one possibility.) Thus, the $C_2C_2C_1C_1 = 4$ possibilities for $\pi$ are:
\begin{align*}
 & 18\ 16\ 14\ 15\ 11\ 12\ \framebox{5\ 6}\ \framebox{3\ 4}\ 7\ 8\ \framebox{2}\ 9\ 10\ 13\ \framebox{1}\ 17\\
 & 18\ 16\ 14\ 15\ 12\ 11\ \framebox{6\ 5}\ \framebox{3\ 4}\ 7\ 8\ \framebox{2}\ 9\ 10\ 13\ \framebox{1}\ 17\\
 & 18\ 16\ 15\ 14\ 11\ 12\ \framebox{5\ 6}\ \framebox{4\ 3}\ 7\ 8\ \framebox{2}\ 9\ 10\ 13\ \framebox{1}\ 17\\
 & 18\ 16\ 15\ 14\ 12\ 11\ \framebox{6\ 5}\ \framebox{4\ 3}\ 7\ 8\ \framebox{2}\ 9\ 10\ 13\ \framebox{1}\ 17\\
\end{align*}

\end{ex}

We are now ready to enumerate $\A_{3n}$.
%Lemmas \ref{MotType} and \ref{CatType} combine to yield the following result.

\begin{prop} \label{CountAn} For all $n \geq 1$,
\[ | \A_{3n} | = \sum_{k=1}^n M_{k-1} \sum_{X \in \mathcal{P}_{n,k}} C_{x_1}C_{x_2}\cdots C_{x_k}\]
where $\mathcal{P}_{n,k}$ is the set of compositions of $n$ of length $k$. Furthermore, the generating function for $| \A_{3n}|$ is given by
\[ A(x) = (c(x)- 1)m(c(x) - 1) \]
where $c(x)$ is the generating function for the Catalan numbers and $m(x)$ is the generating function for the Motzkin numbers.
\end{prop}

\begin{proof} Let $X$ be a composition of $n$ of length $k$. By Lemma \ref{MotType},  there are $M_{k-1}$ Dyck words of type $X$. For each Dyck word of type $X$, there are  $C_{x_1}C_{x_2}\cdots C_{x_k}$ permutations in $\A_{3n}$ by Lemma \ref{CatType}. Because every permutation in $\A_{3n}$ is associated with a Dyck word, summing over all compositions yields the desired formula.

Now, we note that 
\[ (c(x)- 1)m(c(x) - 1) = \sum_{k=1}^\infty M_{k-1} [c(x)-1]^k.\]
However, the coefficient of $x^n$ in $ [c(x)-1]^k$ is simply
\[ \sum_{X \in \mathcal{P}_{n,k}} C_{x_1}C_{x_2}\cdots C_{x_k} \]
and thus the generating function for $\A_{3n}$ is as given. \end{proof}

Up to this point, we have been considering permutations in  $\Av_{3n}^\star(132)$ where all 3-cycles are of the form 312. By considering inverses, similar results hold for permutations in  $\Av_{3n}^\star(132)$ where all 3-cycles are of the form 231. %  In particular, if $\B_{3n} \subseteq  \Av_{3n}^\star(132)$  is the set of permutations on $3n$ elements that avoid the pattern 132 and contain only 3-cycles of the form 231, we have the following results:
%\begin{itemize}
%\item
%If $\pi=\pi_1\pi_2\cdots\pi_n \in \B_{3n}$ and $\pi_j \in \overline{T_2}$ if $\pi_j$ is the ``2" in the 231-pattern realized by its cycle, then $\overline{T_2} = \{1,2, \ldots, n\}.$ (Similar to Definition \ref{def:T} and Lemma \ref{PiConditions})
%\item $|\B_{3n}| = |\A_{3n}|$
%\end{itemize}
We now consider permutations in  $\Av_{3n}^\star(132)$ that contain cycles of both the form 312 and 231. Suppose $\pi \in \Av_{6}^\star(132)$ where $\pi$ consists of one cycle of the form 312 and one cycle of the form 231.  Further suppose that the 312-cycle contains the element 1. It is easy to check that there is only one possibility for $\pi$, namely $\pi=634215$. The cycle of the form 312  is ``615" while the cycle of the form 231 is ``342."  Of course, by considering inverses, there is also only one possibility for $\pi$ if the element 1 is in the cycle of the form 231, namely $\pi = 542361.$

Thus, if a permutation in $\Av_{3n}^\star(132)$ contains a cycle of the form 312 $\alpha = \alpha_3\alpha_1\alpha_2$ and a cycle of the form 231 $\beta=\beta_2\beta_3\beta_1$, these cycles must either be in the relative order
\[ \alpha_3 \beta \alpha_1 \alpha_2 \quad \text{or} \quad \beta_2 \alpha \beta_3 \beta_1.\]
Furthermore, Lemma \ref{PiConditions} states that if a permutation in $\Av_{3n}^\star(132)$ has only cycle of the form 312, the first $n$ elements will be the ``3" from $n$ different cycles of the form 312.  By considering inverses, an analogous result is that if a permutation in $\Av_{3n}^\star(132)$ has only cycle of the form 231, the first $n$ elements will be the ``2" from $n$ different cycles of the form 231. We thus have the following lemma.

\begin{lem}
Let $\pi = \pi_1\pi_2\cdots \pi_{3n} \in \Av_{3n}^\star(132)$. Then $\pi_1, \pi_2, \ldots, \pi_n$ all belong to a different 3-cycle.
\end{lem} 

To count the number of permutations in $\Av_{3n}^\star(132)$, we need only see the cycle type of the first $n$ elements. We are now ready for the main result of this section, the enumeration of $\Av_{3n}^\star(132)$.

\begin{thm} \label{Thm132} For all $n \geq 1$,
\[ |\Av_{3n}^\star(132)| = 2 \sum_{{\mathbf{x}} \in \mathcal{P}_{n,k}}  a_{x_1}a_{x_2}\cdots a_{x_k}\]
where  $\mathcal{P}_{n,k}$ is the set of compositions of $n$ of length $k$ and $a_m = | \A_{3m} |.$ Furthermore, the generating function for $|\Av_{3n}^\star(132)|$ is
\[ B(x) = \frac{2A(x)}{1-A(x)}\] where \[ A(x) = (c(x)- 1)m(c(x) - 1), \] and 
 $c(x)$ and $m(x)$  are the generating functions for the Catalan numbers and Motzkin numbers, respectively.%defined as in Proposition~\ref{CountAn}.
\end{thm}

\begin{proof} For any composition of $X = (x_1, x_2, \ldots, x_k)$ of $n$, we consider a permutation $\pi \in \Av_{3n}^\star(132)$ where the the first $x_1$ elements of $\pi$ are in 312-cycles, the next $x_2$ elements of $\pi$ are in 213-cycles, etc. For each of these blocks of length $x_i$, there are $a_{x_i}$ possible permutations.  If the first $x_1$ elements are in 213-cycles, we get a similar result, and summing over all compositions gives the desired formula. The generating function follows immediately.
\end{proof}

%\todo{define $a_{n,k}$ and explain why its a difference of binomials}

%\todo{prove $T_n = \sum M_{k-1}a_{n,k}$ -- follows from lemmas}

%\todo{Include when there's more than one type of cycle: the tree, the answer for the total}
\section{Avoiding 321}\label{321}

In this section we consider the set $\Av_{3n}^\star(321)$. For a permutation $\pi \in \Av_{3n}^\star(321)$, each of the $n$ 3-cycles corresponds to either a 231 pattern or a 312 pattern. We begin by studying and enumerating those permutations in $\Av_{3n}^\star(321)$ where all $n$ 3-cycles are of the form 312. To this end, let us denote by $\B_{3n} \subseteq \Av_{3n}^\star(321)$ the set of permutations on $3n$ elements that avoid the pattern 321 and contain only 3-cycles of the form 312. Similar to Definition \ref{def:T}, for each $\pi\in\B_{3n}$, we define the sets $T_1, T_2, T_3$, and the word $W$ as below. 

\begin{defn} \label{def:TW} Let $\pi = \pi_1\pi_2\cdots \pi_{3n} \in \B_{3n}$.
\begin{itemize}
\item Define the sets
\begin{align*}
&T_1(\pi) = \{x_1 < x_2 < \cdots < x_n\},\\
&T_2(\pi) = \{y_1 < y_2 <  \cdots < y_n\},  \text{ and}\\
&T_3(\pi) = \{z_1 < z_2 < \cdots < z_n\},
\end{align*} (or simply $T_1, T_2, T_3$) to be a partition of $[3n]$ where $\pi_j \in T_i$ if $\pi_j$ is the ``$i$" in the 312 pattern realized by its cycle, that is, if $\pi_j$ is the $i^{\text{th}}$ smallest number in its cycle. 
\item Define $W(\pi) = w_1w_2\cdots w_{3n}$ to be a word where $w_i = x$ if $\pi_i\in T_1$, $w_i = y$ if $\pi_i \in T_2$, and $w_i =z$ if $\pi_i \in T_3$. 
%or equivalently, if $i= |\{t\in \{j, \pi_j, \pi_{\pi_j}\}: t\leq \pi_j\}|$.
%\item Define $W(\pi)=w_1w_2\cdots w_{3n}$ (or simply $W$) to be a word where $w_i = j$ if $\pi_{i} \in T_j$ for $1 \leq i \leq 3n$ and $1\leq j\leq 3$.
\end{itemize}
\end{defn}

Suppose $\pi \in \B_6$. It is easy to check that there are only three possibilities for $\pi$, namely 
\begin{align*}
\pi &= 561234,\\
\pi &= 416235, \text{ or}\\
\pi &= 312645.
\end{align*}
Note that in all three cases, $\pi$ can be written in its cycle notation as $\pi = (z_1, x_1, y_1)(z_2, x_2, y_2)$. For a general permutation $\pi \in \B_{3n}$, any two 3-cycles must be in the same relative order as the three possibilities in $\B_6$. Thus, we state the following observation.

\begin{obs} \label{obs321}
If $\pi \in \B_{3n}$, and $T_1, T_2,$ and $T_3$ are as defined in Definition \ref{def:TW}, then $\pi$ can be written as\[ \pi = \prod_{i=1}^{n} (x_i, z_i, y_i). \]
\end{obs}

We can now state some properties of $T_1(\pi), T_2(\pi),$ and $T_3(\pi)$.

%\todo{Show that $x_i, y_i, z_i$ are part of the same cycle so that part (a) of the next lemma makes sense.}

\begin{lem}\label{lem:Lem321}
Let $\pi = \pi_1\pi_2\cdots \pi_{3n} \in \B_{3n}$ and $T_1(\pi),$ $T_2(\pi)$, and $T_3(\pi)$ be as defined in Definition~\ref{def:TW}.
%\begin{align*}
%&T_3(\pi) = \{x_1 < x_2 < \cdots < x_n\},\\
%&T_1(\pi) = \{y_1 < y_2 <  \cdots < y_n\},  \text{ and}\\
%&T_2(\pi) = \{z_1 < z_2 < \cdots < z_n\}.
%\end{align*}
%Then in its cycle notation, $\pi$ is written as
%\[ \pi = \prod_{i=1}^{n} (z_i, y_i, x_i). \]
Then, for $i \in [n]$ and $j \in [n-1]$,
\begin{enumerate}[label=(\alph*)]
\item\label{item:a} $x_i < y_i < z_i$,
\item\label{item:b} $x_i \leq 3i - 2$, and
\item\label{item:c} $x_j < y_i < x_{j+1}$ if and only if $y_j < z_i < y_{j+1}$. (Equivalently, in the word $W(\pi)$, the $i^{\text{th}}$ $x$ appears between the $j^{\text{th}}$ and $(j+1)^{\text{st}}$ $z$ if and only if the $i^{\text{th}}$ $y$ appears between the $j^{\text{th}}$ and $(j+1)^{\text{st}}$ $x$.)
\end{enumerate}
Conversely, given sets $T_1$, $T_2$, and $T_3$ satisfying conditions \ref{item:a}, \ref{item:b}, and \ref{item:c}, there exists a unique permutation 
$\pi \in \B_{3n}$ with $T_1(\pi) = T_1$, $T_2(\pi)=T_2$, and $T_3(\pi) = T_3$, namely, 
\[ \pi = \prod_{i=1}^{n} (x_i, z_i, y_i). \] 
\end{lem}

%\todo{Add another lemma saying that if you have sets $T_i$ satisfying the conditions above, then you can build a permutation $\pi$ in $\B_{3n}$.}

\begin{proof}
Since each cycle of $\pi$ is of the form 312, \ref{item:a} holds. For $i \in [n]$, there are at most $3(i-1)$ elements before $\pi_{x_i}$, so $x_i \leq 3i - 2$ and \ref{item:b} is true.

To prove part \ref{item:c}, we let $i \in [n]$ and pick $j \in [n-1]$ with $x_j < y_i < x_{j+1}$. First, assume toward a contradiction that $z_i < y_j$. In this case, the subsequence $z_iy_ix_j$ occurs in $\pi$ since the positions of the subsequence are $x_i$, $z_i$, and $y_j$. This subsequence is a 321-pattern since $z_i > y_i$ by part \ref{item:a}. Now assume toward a contradiction that  $z_i > y_{j+1}$. In this case, the subsequence $z_{j+1}x_{j+1}y_i$ is a 321-pattern.

For the other direction of part \ref{item:c}, we start by letting $i \in [n]$ and pick $j \in [n-1]$ with $y_j < z_i < y_{j+1}$. Then, assume toward a contradiction that $y_i < x_j$. In this case, $z_jx_jy_i$ is a 321-pattern. If we assume toward a contradiction that $y_i > x_{j+1}$, then $z_iy_ix_{j+1}$ is a 321-pattern. Thus part \ref{item:c} holds.  Finally, the converse follows from Observation \ref{obs321}.

\end{proof}

We illustrate Definition \ref{def:TW} and  Lemma~\ref{lem:Lem321} with the following example.
\begin{ex}\label{ex:Ex321}
Let 
\begin{align*}
\pi& = 7\ 8\ 1\ 2\ 11\ 12\ 3\ 4\ 5\ 6\ 9\ 10\ 16\ 13\ 18\ 14\ 15\ 17\\&= (1,7,3)(2,8,4)(5,11,9)(6,12,10)(13,16,14)(15,18,17).
\end{align*}
Then
\begin{align*}
 T_1 &= \{1,2,5,6,13,15\},\\
T_2 &= \{3,4,9,10,14,17\},\\
  T_3 &=\{7,8,11,12,16,18\},
\end{align*}
and $W(\pi) = zzxxzzyyxxyyzxzyxy$.

Parts \ref{item:a} and \ref{item:b} of Lemma \ref{lem:Lem321} are straightforward to see. To illustrate part \ref{item:c}, 
%we note that $y_1=3$ and $y_2 = 4$ and thus
%\[ x_2 < y_1, y_2 < x_3.\] Furthermore, $z_1=7$ and $z_2 = 8$ which are both between $y_2$ and $y_3$.  Similarly,
%\begin{itemize}
%\item $x_4 < y_3, y_4 < x_5$ and $y_4 < z_3, z_4 < z_5$; and
%\item $x_5 < y_5 < x_6$ and $y_5 < z_5 < y_6$.
%\end{itemize}
notice that in the word $W(\pi)$, the $1^{\text{st}}$ and $2^{\text{nd}}$ $x$ are between the $2^{\text{nd}}$ and $3^{\text{rd}}$ $z$ and therefore the $1^{\text{st}}$ and $2^{\text{nd}}$ $y$ are between the $2^{\text{nd}}$ and $3^{\text{rd}}$ $x$.
\end{ex}

It turns out that all three sets are not needed to determine a unique permutation.  In fact, given any set $T$ of $n$ positive integers satisfying only condition \ref{item:b} in Lemma \ref{lem:Lem321}, there exists a unique corresponding permutation in $\B_{3n}$.

\begin{defn}
Given word $W=w_1\cdots w_{k}$, let %and $w_k\neq z$, let
\begin{align*}
X(W) &= \text{the number of } x\text{'s in } W \text{, and}\\
Y(W) &= \text{the number of } y\text{'s in } W.
\end{align*}
\end{defn}

\begin{lem}\label{FlatsDetermine}
Given a set $T = \{t_1 < t_2 < \cdots < t_n\}$ with $t_i \leq 3i - 2$, there is a unique permutation $\pi \in \B_{3n}$ with $T_1(\pi) = T$.
\end{lem}

\begin{proof} 
Let $S = [3n] \setminus T = \{s_1 < s_2 < \cdots <s_{2n}\}$. This proof will involve letting $T_1=T$ and finding unique sets $T_2$ and $T_3$ so that $T_1, T_2$, and $T_3$ satisfy the properties listed in Lemma \ref{lem:Lem321}.

To begin, we will construct a word $W(T) = w_1w_2\ldots w_{3n}$ on the alphabet $\{x,y,z\}$. The positions of these $x$'s and $y$'s will determine the sets $T_2$ and $T_3$. We first let $w_i = z$ for all $i\in T$. To determine the other values, we let $w_{s_1} = x$ and proceed by induction. 

Suppose $W'=w_1\cdots w_{k-1}$, and $w_k\neq z$.
If $X(W') = Y(W')$, set $w_k = x$. Otherwise $X(W')>Y(W')$. Let $i=Y(W')+1$ and let $j$ satisfy the condition that exactly $j$ occurrences of $z$ appear before the $i^{\text{th}}$ $x$. If $j = X(W')$, then $w_k=y$; otherwise we write $w_k=x$. We stop when the length of $W(T)$ is $3n$.

We now create the sets $T_2$ and $T_3$ based on the word $W(T)$:
\[ T_2 = \{i \mid w_i = x\} \] \[T_3 = \{i \mid w_i = y\}\] and thus $\pi$ is uniquely determined by this algorithm and $W(\pi) = W(T)$.
Note that the sets $T_1, T_2,$ and $T_3$ clearly satisfy the conditions of Lemma \ref{lem:Lem321} and so $\pi \in \B_{3n}$.
\end{proof}

\begin{ex}

Let $T = \{1,2,3,6,11,14\}$ which satisfies the hypothesis of Lemma \ref{FlatsDetermine}.  To find the corresponding permutation $\pi \in \B_{18}$, we first set $S = [18]\setminus T= \{4, 5, 7, 8, 9, 10, 12, 13, 15, 16, 17, 18\}$. Following the algorithm in the proof of Lemma \ref{FlatsDetermine}, we then set $w_i=z$ for all $i \in T$ and set $w_4 = x$. To figure out what $w_5$ should be, notice that $X(zzzx) = 1$ and $Y(zzzx) = 0$, so we set $i=1$. Then notice that exactly $j=3$ $z$'s appear before the $i=1^{\text{st}}$ occurrence of $x$. Since $j \neq X(zzzx)$, set $w_5 = x$. We continue with this algorithm to get
\[ w(T)= zzzxxzxyyxzyyzxxyy. \]
This yields
\[ T_2 = \{4, 5, 7, 10, 15, 16\}, \]
\[ T_3 = \{8, 9, 12, 13, 17, 18\}, \]
which combine to yield the permutation
\begin{align*}
\pi &= (1, 8, 4)(2, 9, 5)(3, 12, 7)(6, 13, 10)(11, 17, 15)(14, 18, 16)\\
&= 8\ \ 9\ \ 12\ \ 1\ \ 2\ \ 13\ \ 3\ \ 4\ \ 5\ \ 6\ \ 17\ \ 7\ \ 10\ \ 18\ \ 11\ \ 14\ \ 15\ \ 16.
\end{align*}

\end{ex}

Lemmas \ref{lem:Lem321} and \ref{FlatsDetermine} show there a one-to-one correspondence between sets of the form \[T = \{t_1 < t_2 < \cdots < t_n\} \] satisfying $t_i \leq 3i -2$ and the set of permutations $\B_{3n}$. We enumerate the sets $T$ as characterized above in the proof of the next proposition.

\begin{prop} For all $n \geq 1$,
\[ |\B_{3n}| = \frac{1}{2n+1} {3n \choose n}. \]
\end{prop}

\begin{proof}
The numbers given by the right hand side of the stated equation are a special case of the Fuss-Catalan numbers and are well known to enumerate generalized Dyck paths from $(0,0)$ to $(2n,n)$ (see for example \cite{Aval}). These paths are comprised of only east-steps and north-steps and always lie below the line $y = \frac{1}{2}x$. There is a straightforward bijection between these paths and sets of the form $T = \{t_1 < t_2 < \cdots < t_n\}$ with $t_i \leq 3i - 2$, described here.

Given a generalized Dyck path $D$ as described above, there are $2n$ east-steps. The height of the $i^{\text{th}}$ east-step must be at a height between 0 and $2(i-1)$ in order to satisfy the requirement that the path lies under the line $y = \frac{1}{2}x$. Thus we can obtain a sequence $S = \{s_1 < s_2 < \cdots < s_n\}$ with $0\leq s_i \leq 2i-2$. This is clearly invertible since the heights of the east-steps of $D$ determine the location of the north-steps as well. We obtain the sequence $T= \{t_1 < t_2 < \cdots < t_n\}$ by letting $t_i = i+s_i$. Since this is a bijection, the sets $T$ of this form are also enumerated by $\frac{1}{2n+1} {3n \choose n}$. By Lemma~\ref{FlatsDetermine}, the proposition follows. 
\end{proof}

\begin{cor}\label{cor-flats}
Let $\mathcal{C}_{3n}$ be the set of all permutations in $Av_{3n}^\star(321)$ composed of only cycles of the form 231. Then for all $n \geq 1$,
\[ |\mathcal{C}_{3n}| = \frac{1}{2n+1} {3n \choose n}. \]
\end{cor}
\begin{proof}
Notice that if $\pi$ avoids 321, we must have that $\pi^{-1}$ avoids 321 as well since 321 is its own inverse. Since $312^{-1} = (132)^{-1} = (123) = 231$, each cycle of the form 312 becomes a cycle of the form 213 when we take its inverse, and thus the result follows. 
\end{proof}

The goal of this section is to enumerate all 321-avoiding permutations composed of only $3$-cycles. To this end, we need to consider permutations that contain both cycles of the form 231 and cycles of the form 132. We proceed with the following definition.

\begin{defn}
Let $n\geq 1$.
\begin{itemize}
\item Let $\mathcal{T}_n$ be the set of all subsets $T = \{t_1<t_2<\cdots<t_n\}$ of $[3n]$ satisfying $t_i<3n-2$. 
\item Let $W(T)$ be defined as in the proof of Lemma~\ref{FlatsDetermine}.
\item For a word $W = w_1w_2\ldots w_{3n}$, let $W^{(i)}$ be the $i^{\text{th}}$ \textit{initial segment} of $W$ defined to be $W^{(i)} = w_1w_2\ldots w_j$ where $w_j$ is the $i^{\text{th}}$ occurrence of $y$ in $W$.
\end{itemize}
\end{defn}

\begin{ex}\label{ex n=2}
Consider the case when $n=2$. Then $\mathcal{T}_2$ consists of the sets $ \{1,2\},  \{1,3\},$ and $\{1,4\}$ corresponding to the words \begin{align*}
W(\{1,2\}) &= zzxxyy, \\
W(\{1,3\}) &=zxzyxy, \\
W(\{1,4\}) &=zxyzxy ,
\end{align*}
respectively. As before, the $i^{\text{th}}$ occurrence of $x$, $i^{\text{th}}$ occurrence of $y$, and the $i^{\text{th}}$ occurrence of $z$ are all in the same cycle. 
Consider the first word $z_1z_2x_1x_2y_1y_2$. If we let $z_1x_1y_1$ be a 312 pattern and let $z_2x_2y_2$ be a 231 pattern, then we get $\pi=541632$ which contains the $321$-pattern $x_2y_1y_2 = 632$. Similarly, if we let $z_1x_1y_1$ be a 231 pattern and let $z_2x_2y_2$ be a 312 pattern, then we get $\pi=365214$ which contains the $321$ pattern $x_1x_2y_1 = 521$. Thus for the word $z_1z_2x_1x_2y_1y_2$ (which has $x_2$ appear before $y_1$), we must have that both cycles have the same pattern type.

For the other two words, we find that the permutations obtained by letting  $z_1x_1y_1$ be a 312 pattern and let $z_2x_2y_2$ be a 231 pattern or by letting  $z_1x_1y_1$ be a 231 pattern and let $z_2x_2y_2$ be a 312 pattern do avoid the pattern 321. In particular we get the four permutations 
$$ 246135, 415263, 231645, 312564.$$
As we've seen in Lemma~\ref{FlatsDetermine} and Corollary~\ref{cor-flats}, for each of these three words, we can also take all cycles to be of the same type, so we have that $|\Av_{6}^\star(321)| = 2+4+4 = 10.$
\end{ex}

\begin{prop}\label{prop:all 321}
For $n\geq 1$, we have that
$$ |\Av_{3n}^\star(321)| = \sum_{T \in \mathcal{T}_n} 2^{h(T)}$$
where $h(T)$ is the number of times  $X(W^{(i)})=i$ for initial segments $W^{(i)}$ of the word $W=W(T)$.
\end{prop}

\begin{proof}
Let $T\in\mathcal{T}_n$ and let $W = W(T)$. Consider the set $M= \{m_1< m_2< \ldots< m_h\}$ where $X(W^{(m_i)}) = m_i$, i.e.\ for each $i$ the $m_i^{\text{th}}$ initial segment satisfies the condition that the number of $x$'s is equal to the number of $y$'s. Notice $m_h=n$ and $h(T) = h$. We will show that the $r^{\text{th}}$ cycle and the $s^{\text{th}}$ cycle have to be of the same form if and only if $m_{i-1}<r,s\leq m_{i}$ with the convention that $m_0:=0$. Since each cycle must be of the form 312 or 231, then this leaves $2^{h(T)}$ possible choices for each set $T$. 

We will first show that for all $m_{i-1}<r,s\leq m_{i}$, we must have that the cycle containing elements of $\{x_r, y_r, z_r\}$ must be of the same type as the cycle containing $\{x_s, y_s, z_s\}$. In other words, the cycles associated to the first $m_1$ $z$'s must either all be cycles of the form 312 or all be cycles of the form 231. Then the cycles associated to the next $(m_2-m_1)$ $z$'s must either all be cycles of the form 312 or all be cycles of the form 231, and so on. It is enough to show that for any $m_{i}<r <m_{i+1}$ the $r^{\text{th}}$ cycle and $(r+1)^{\text{st}}$ cycle are associated to the word $z_rz_{r+1}x_rx_{r+1}y_ry_{r+1}$ by Example~\ref{ex n=2}. But this must be true by the definition of $M$ because if $x_{r+1}$ appeared after $y_r$, we would have that $X(W^{(r)}) = r$. 

We now will show that for each $r\leq m_i$ and $s>m_i$, we cannot have that the $r^{\text{th}}$ cycle and $s^{\text{th}}$ cycle appear in the permutation in the form $z_rz_sx_rx_sy_ry_s$. Indeed, if $r=m_i$, then this would violate the condition that $X(W^{(m_i)}) = m_i$ since there are at least $s>m_i$ $x$'s appearing before the $m_i^{\text{th}}$ $y$. If $r<m_i$, then we must have the subword $z_rz_{m_i}z_sx_rx_{m_i}x_sy_ry_{m_i}y_s$ which also indicates that $X(W^{(m_i)}) \neq m_i$ since there are at least $s>m_i$ $x$'s appearing before the $m_i^{\text{th}}$ $y$. Therefore the subword given by the $r^{\text{th}}$ and $s^{\text{th}}$ cycles cannot contribute a 321-pattern even if they are of different forms. 

We have shown that for $r\leq m_i$ and $s>m_i$ and the $r^{\text{th}}$ cycle is of a different form than the $s^{\text{th}}$ cycle, the subword of $\pi$ given by those cycles does not contain a 321 pattern. It remains to show that for $r<s<t$ with $r\leq m_i$ and $t>m_i$, there is no 321 pattern appearing in the subword of $\pi$ given by those cycles. There are six possible subwords containing 3 cycles where $h=2$:
\[ zzxxzyyxy, zxzzyxxyy, zxzyzxxyy, zxyzzxxyy, zzxxyzyxy, zzxxyyzxy.\]
If the first cycle is a cycle of the form 312, and the last cycle is a cycle of the form 231, it is straightforward to check that the possible permutations do not contain a 321 pattern.  There are also four possible subwords containing 3 cycles where $h=3$:
\[ zxzyxzyxy, zxzyxyzxy, zxyzxzyxy, zxyzxyzxy.\]
Again, if the first cycle is a cycle of the form 312, and at least one of the other two cycles is a cycle of the form 231, it is straightforward to check that the possible permutations do not contain a 321 pattern.%Since in the previous paragraph, we showed this for pairs of cycles, it is enough to show that there is no 321 pattern $\pi_i\pi_j\pi_k$ where each element is in a different cycle. %If the $r$-th, $s$-th, and $t$-th cycle are of the same pattern, then this follows from Lemma~\ref{lem:Lem321}.
\end{proof}

%\todo{Finish this last paragraph -- three interweaving cycles. Of the 12 possible configurations of three 3-cycles, in order to have a 321 pattern with an element from each cycle, we must have either 123123123, 121321323, 121231233, or 112312323. Those first two have all three cycles in the same block, so it remains to check this for the last two.
%
%We should say "it is straightforward to check that ...".}

We can rewrite the previous proposition in terms of Dyck words of semilength $n$.

\begin{thm} \label{Thm321}
For all $n \geq 1$,
\[ |\Av_{3n}^\star(321)| = \sum_{D \in \D_n} 2^{h(D)} \prod_{i=1}^{n-1} {r_i(D) + s_i(D) \choose r_i(D)} \]
where 
\begin{itemize}
\item $\D_n$ is the set of Dyck words on $\{x,y\}$ of semilength $n$, 
\item $h(D)$ is the number of times an initial segment of $D$ has equal number of $x$'s and $y$'s, 
\item $r_i(D)$ is the number of $y$'s between the $i^{\text{th}}$ $x$ and the $(i+1)^{\text{st}}$ $x$, and 
\item $s_i(D)$ is the number of $x$'s between the $i^{\text{th}}$ $y$ and the $(i+1)^{\text{st}}$ $y$.
\end{itemize}
\end{thm}

\begin{proof}[Proof of Theorem \ref{Thm321}]
Let $T\in\mathcal{T}_n$. First notice that by construction, the subword $D_T$ of $W(T)$ of length $2n$ given by taking all occurrences of $x$ and $y$ is a Dyck word. 
By Proposition~\ref{prop:all 321}, it suffices to show that for any given Dyck word $D$, the number of sets $T\in\mathcal{T}_n$ with $D_T=D$ is $\prod_{i=1}^{n-1} {r_i(D) + s_i(D) \choose r_i(D)}.$

Since we already have the relative order of $x$’s and $y$’s, we need only to decide where the $z$’s are in relation to these $x$’s and $y$’s. Consider a fixed $1 \leq i \leq n$.  We first claim that there are $s_i(D)$ $z$’s between the $i^{\text{th}}$ and $(i+1)^{\text{st}}$ $x$.  Since there are $s_i$ $x$’s between the $i^{\text{th}}$ and $(i+1)^{\text{st}}$ $x$ (by definition), there is some $j$ so that:
\begin{itemize}
\item the $i^{\text{th}}$ $y$ appears between the  $j^{\text{th}}$ and $(j+1)^{\text{st}}$ $x$, and
\item the $(i+1)^{\text{st}}$ $y$ appears between the $(j+s_i)^{\text{th}}$ and $(j+s_i + 1)^{\text{st}}$ $x$.
\end{itemize}
By Lemma~\ref{lem:Lem321} (c), this is equivalent to:
\begin{itemize}
\item the $i^{\text{th}}$ $x$ appears between the  $j^{\text{th}}$ and $(j+1)^{\text{st}}$ $z$, and
\item the $(i+1)^{\text{st}}$ $x$ appears between the $(j+s_i)^{\text{th}}$ and $(j+s_i + 1)^{\text{st}}$ $z$.
\end{itemize}
Thus, there are $s_i$ $z$’s between the $i^{\text{th}}$ and $(i+1)^{\text{st}}$ $x$ -- namely, the $(j+1)^{\text{st}}$ $z$ through the $(j+s_i)^{\text{th}}$ $z$ -- and our claim holds.

Now consider our word $D$ again. Between the $i^{\text{th}}$ and $(i+1)^{\text{st}}$ $x$, there must be $r_i$ $y$’s and $s_i$ $z$’s.  There are ${r_i + s_i \choose r_i}$ such subwords of $y$’s and $z$’s. We note as well that if $k$ is the number so that $r_i=0$ for $i < k$ (with the assumption that $r_0 = 0$), then there must be $k$ $z$’s before the first $x$. 
Since any such words will satisfy the conditions of Lemma~\ref{lem:Lem321}, we have that the set of words constructed from $D$ is $W(T)$ for some $T$ and thus there are $\prod_{i=1}^{n-1} {r_i(D) + s_i(D) \choose r_i(D)}$ sets $T\in\mathcal{T}_n$ associated to $D$. Since each $T$ is associated to a unique Dyck path $D$, the result follows.
\end{proof}

We illustrate this proof with an example. 
\begin{ex}
Suppose $D = xxyyxyxxyxyy$. Then we have $h(D) = 3$, $\mathbf{r}(D)= (0,2,1,0,1)$, and $\mathbf{s}(D) = (0,1,2,1,0)$. Then $\prod_{i=1}^{n-1} {r_i(D) + s_i(D) \choose r_i(D)}=9$.  

Following the proof of Theorem \ref{Thm321}, there must be two $z$’s before the first $x$.  Also, between the first and second $x$, there must be $r_1 = 0$ $y$’s and $s_1=0$ $z$’s.  Between the second and third $x$, there must be $r_2 = 2$ $y$’s and $s_2=1$ $z$’s; the possible subwords of $y$ and $z$ here are:
\[ yyz, yzy, zyy.\]
Then, between the third and fourth $x$, there must be $r_3=1$ $y$’s and $s_3=2$ $z$’s; the possible subwords of $y$ and $z$ here are:
\[ yzz, zyz, zzy. \]
Finally, there are 0 $y$’s and 1 $z$ between the fourth and fifth $x$, and 1 $y$ and 0 $z$’s between the fifth and sixth $x$.  Thus, the 9 possible words $W$ are:
\begin{align*} &zzxxyyzxyzzxzxyxyy, zzxxyyzxzyzxzxyxyy, zzxxyyzxzzyxzxyxyy,\\& zzxxyzyxyzzxzxyxyy, zzxxyzyxzyzxzxyxyy, zzxxyzyxzzyxzxyxyy, \\ &zzxxzyyxyzzxzxyxyy, zzxxzyyxzyzxzxyxyy, zzxxzyyxzzyxzxyxyy.
\end{align*}
Notice that by considering the positions of the $z$’s, these correspond to the sets in $\mathcal{T}_6$:
\begin{align*}
& \{1,2,7,10,11,13\}, \{1,2,7,9,11,13\},\{1,2,7,9,10,13\},\\ &\{1,2,6,10,11,13\},\{1,2,6,9,11,13\},
\{1,2,6,9,10,13\},\\ &\{1,2,5,10,11,13\},\{1,2,5,9,11,13\},\{1,2,5,9,10,13\}.
\end{align*}
Because $h(D)=3$, there are $2^3=8$ permutations in $\Av^\star_{18}(321)$ associated to each of the nine sets above.  In particular, the first two 3-cycles are both either of the form 231 or 312; the third 3-cycle is either of the form 231 or 312; and the last three 3-cycles are all of the same form. We illustrate this by explicitly listing the eight permutations associated to the set $T=\{1,2,7,10,11,13\}$ below:
\begin{align*}
& 5 \  6 \ 1 \  2 \ 3 \ 4 \ 9 \ 7 \ 8 \ 15 \ 17 \ 10 \ 18 \ 11 \ 12 \ 13 \ 14  \ 16 
& 5 \  6 \ 1 \  2 \ 3 \ 4 \ 9 \ 7 \ 8 \ 12 \ 14 \ 15 \ 16 \ 17 \ 10 \ 18 \ 11 \ 13\phantom{.} \\
& 5 \  6 \ 1 \  2 \ 3 \ 4 \ 8 \ 9 \ 7 \ 15 \ 17 \ 10 \ 18 \ 11 \ 12 \ 13 \ 14  \ 16 
& 5 \  6 \ 1 \  2 \ 3 \ 4 \ 8 \ 9 \ 7 \ 12 \ 14 \ 15 \ 16 \ 17 \ 10 \ 18 \ 11 \ 13\phantom{.} \\
& 3 \ 4 \ 5 \ 6 \ 1 \ 2  \ 9 \ 7 \ 8 \ 15 \ 17 \ 10 \ 18 \ 11 \ 12 \ 13 \ 14  \ 16 
& 3 \ 4 \ 5 \ 6 \ 1 \ 2 \ 9 \ 7 \ 8 \ 12 \ 14 \ 15 \ 16 \ 17 \ 10 \ 18 \ 11 \ 13\phantom{.}\\
& 3 \ 4 \ 5 \ 6 \ 1 \ 2  \ 8 \ 9 \ 7 \ 15 \ 17 \ 10 \ 18 \ 11 \ 12 \ 13 \ 14  \ 16 
& 3 \ 4 \ 5 \ 6 \ 1 \ 2 \ 8 \ 9 \ 7 \ 12 \ 14 \ 15 \ 16 \ 17 \ 10 \ 18 \ 11 \ 13. 
\end{align*}
\end{ex}

As a corollary of the Proof of Theorem~\ref{Thm321}, we have an interesting property of Dyck paths. 

\begin{cor}
For all $n \geq 1$,
\[ \sum_{D \in \D_n}  \prod_{i=1}^{n-1} {r_i(D) + s_i(D) \choose r_i(D)} = \frac{1}{2n+1}{{3n}\choose{n}} \] 
where $\D_n$ are the Dyck words of semilength $n$ and $r_i(D)$ and $s_i(D)$ are defined as in Theorem~\ref{Thm321}.
\end{cor}

We conclude this section by asking the following question. 
\begin{question}
Consider the polynomial 
\[f(t) = \sum_{D \in \D_n} t^{h(D)} \prod_{i=1}^{n-1} {r_i(D) + s_i(D) \choose r_i(D)}. \] Is there a way to extract the coefficients of $t^k$ in general?
\end{question}

Notice that $f(1) = \frac{1}{2n+1}{3n \choose n}$ and $f(2) = |\Av^\star_{3n}(321)|$. Answering the above question would give us a closed form for $|\Av^\star_{3n}(321)|$.

\section{Avoiding $123$ or a pair of patterns}\label{other}

In this last section, we enumerate the permutations in $S^\star_{3n}$ that avoid 123 or any pair of patterns of length 3.

\begin{thm}\label{Thm123}
For all $n \geq 3$,  $$|\Av_{3n}^\star(123)| = 0.$$ Also, $|\Av_{3}^\star(123)| = 2$ and $|\Av_{6}^\star(123)| = 6$. 
\end{thm}

\begin{proof}
It is easily verified that $|\Av_{9}^\star(123)| = 0$. If there were some $n>3$ so that $\pi\in \Av_{3n}^\star(123)$, taking any three 3-cycles of $\pi$ and rescaling them would give a permutation in $\Av_{9}^\star(123)$. Thus $|\Av_{3n}^\star(123)| = 0$ for all $n\geq 3$. 
\end{proof}

\begin{thm} For $n\geq 1$, we have 
 $$|\Av_{3n}^\star(132, 213)| =|\Av_{3n}^\star(132, 321)| =2$$
 and
  $$|\Av_{3n}^\star(132, 231)| =|\Av_{3n}^\star(231,321)| =1.$$
 For all pairs not equivalent to the ones listed above via inverse or reverse-complement, there are no permutations composed of only 3-cycles that avoid that pair when $n>2$. 
\end{thm}

\bibliographystyle{amsplain}

\end{document}